\documentclass[12pt]{amsart}

\usepackage{amssymb,amsmath,graphicx,color,amsthm,bbm, bbold, mathtools}

\definecolor{Red}{cmyk}{0,1,1,0}

\definecolor{verde}{cmyk}{1,0,1,0}

\definecolor{azul}{cmyk}{1,1,0,0}

\numberwithin{equation}{section}

\def\Ed{{\mathbb{E}}}
\def\Pd{{\mathbb{P}}}

\newcommand{\N}{\mathbb{N}}

\newcommand{\e}{\varepsilon}

\newtheorem{theorem}{Theorem}
\newtheorem{proposition}[equation]{Proposition}
\newtheorem{definition}{Definition}
\newtheorem{lemma}{Lemma}
\newtheorem{corollary}[equation]{Corollary}

\begin{document}

\title{Agglomeration in a preferential attachment random graph with edge-steps}

\author{Caio Alves$^1$}
\address{$^1$ University of Leipzig, Germany. \newline
	e-mail: \textit{caio.alves@math.uni-leipzig.de}}

\author{Rodrigo Ribeiro$^2$}
\address{$^2$ Pontificia Universidad Cat\'{o}lica de Chile, Chile.
	\newline
	e-mail: \textit{rribeiro@impa.br}}

\author{R{\'e}my Sanchis$^3$}
\address{$^3$ Universidade Federal de Minas Gerais, Brazil
	\newline
	e-mail: \textit{rsanchis@mat.ufmg.br}}

\begin{abstract}In this paper we investigate geometric properties of graphs generated by a preferential attachment random graph model with \textit{edge-steps}. More precisely, at each time~$t\in\N$, with probability $p$ a new vertex is added to the graph (a \textit{vertex-step} occurs) or with probability~$1-p$ an edge connecting two existent vertices is added (an \textit{edge-step} occurs). We prove that the \textit{global clustering coefficient} decays as $t^{-\gamma(p)}$ for a positive function $\gamma$ of $p$. We also prove that the \textit{clique number} of these graphs is, up to sub-polynomially small factors, of order~$t^{(1-p)/(2-p)}$.
\end{abstract}

\maketitle

\section{Introduction}
Empirical findings on properties of concrete networks have encouraged the proposal and investigation of non-homogeneous random graph models. The data obtained from complex networks coming from distinct contexts has suggested that, although different in background, those networks share many special properties such as scale-freeness and small diameter. In this paper, we are interested in the fact that such networks are highly clustered. We do not intend to survey the enormous amount of work done in the field, but the interested reader may find in \cite{CLBook, DBook,HBook} some important rigorous results about many properties of different models investigated so far and in \cite{BA99, SW98} a vast set of empirical properties found.

As an attempt of producing a scale-free graphs, in the seminal work \cite{BA99}, R. Albert and A. Barab\'{a}si proposed a dynamical random graph model in which at each step a new vertex is added with $m$ edges emanating from it. Its $m$ neighbors are then independently chosen from the previous vertex-set with probability proportional to their degrees. This rule of attachment, known as \textit{preferential attachment}, proved itself efficient to produce graphs having many of the properties observed empirically. However, this important model does not produce graphs with large cliques neither high density of triangles \cite{bollrior}. Another downside is the fact that the dynamic proposed in \cite{BA99} forbids connections between already existent vertices, i.e., given two vertices in the graph, they may be connected only at the time the later vertex is added. In terms os real-life networks, connections between older vertices are frequently natural and expected.

In this paper we study a modification of the traditional Barab\'{a}si-Albert's model (BA-model) in which connections between already existent vertices are allowed. This mechanism is present in a very general model proposed by Cooper and Frieze in \cite{CF03}, which under certain choices for the parameter also produce scale-free graphs. Roughly speaking, our aim is to investigate the impact on structures of these graphs when this special kind of connection is allowed.
\subsection{A preferential attachment dynamic with edge-steps}
The model here investigated is defined inductively. At each step we decide according to a specific rule how to obtain the new graph from the previous one. There are two ways in which we modify the graphs:
\begin{itemize}
	\item \textit{Vertex-step}: We add a new vertex $v$ to the graph $G$ and connect $v$ to a vertex~$u$ in $G$ selected according the \textit{preferential attachment rule}, i.e., $u$ is selected with probability
	\begin{equation}\label{eq:PArule}
	\Pd \left(  u \text{ is chosen}\; \middle | \;G \right) = \frac{\text{degree of } u \text{ in }G}{\text{sum of the degrees of all vertices in } G };
	\end{equation}
	\item \textit{Edge-step}: A new edge $\{u,w\}$ is added to $G$, where $u$ and $w$ are vertices in~$G$ chosen independently and also according to the preferential attachment rule above described.
\end{itemize}
We point out that, in the edge-step, the vertices $u$ and $w$ may be the same, in this case we add a \textit{loop}. Moreover, $u$ and $w$ may be already connected, in this case we allow \textit{multiple edges} in the process.

The model evolves as follows: Given a parameter $p \in [0,1]$, consider an initial graph~$G_1$ and a collection of \textit{i.i.d.} random variables~$\{ Z_t\}_{t \ge 2}$ following a Bernoulli's distribution with parameter $p$. For each integer~$t\geq 2$ we obtain $G_{t+1}$ from $G_t$ by  performing either a \textit{vertex-step} on $G_t$ if $Z_{t+1} = 1$, or an \textit{edge-step} otherwise. In this setting, we let $\mathcal{F}_t$ denote the $\sigma$-algebra encoding all our knowledge about the process up to time $t$.

We observe that, when $p = 1$, this model correspond to the BA-model with $m=1$.  Throughout the paper we let $G_1$ be the graph with one vertex and one loop attached to it. This choice is made exclusively to simplify the expressions of the probabilities we have to deal with and has no lasting impact on the graph's asymptotic structure.
\subsection{Clustering coefficient} One of the common feature of many concrete networks is clustering, i.e., tendency that ``\textit{people with common friends tend to become friends}". One way of quantifying this tendency of \textit{closing triangles} is the \textit{global clustering coefficient (or transitivity)}, $\tau(G)$, which is defined as
\begin{equation}
\tau(G) := 3 \times\frac{ \# \text{ triangles in }G}{\# \text{ paths of length 2 in }G}.
\end{equation}
The observable $\tau(G)$ measures the probability of a uniformly chosen pair of vertices that have a common neighbor being connected. 

In \cite{bollrior}, for the traditional BA-model with $m\ge 2$, the authors showed that~$\mathbb{E}[\tau(G_t)]$ decays as~$\log^2(t)/t$ and the expected number of triangles at time $t$ is of order~$\log^3(t)$. If one desires to increase the global clustering one may try to increase the number of triangles, even if this also increases the number of paths of length $2$. In this direction, it is reasonable to expect that the majority of the triangles is formed by those vertices of high degree, since all the connections following the PA-rule are made with probability proportional to the degree of the vertices. Therefore, if the probability of choosing vertices of small degree is decreased (\textit{respec. increased}), somehow one may expect a higher (\textit{respec. smaller}) number of triangles. One way of achieving this is known as the \textit{affine} preferential attachment rule, \cite{M05}. In this scheme a constant of \textit{attractiveness}, $\delta$,  is introduced in the rule (\ref{eq:PArule}) in such way that a new vertex $v$ chooses a neighbor $u$ with probability proportional to the degree of $u$ plus $\delta$.
In this setup, a positive~$\delta$ gives an ``extra" chance of a small degree vertex being chosen. In \cite{EN11} the authors showed that for any positive~$\delta$ the expected value of the number of triangles decreases and is of order~$\log^2(t)$ and ~$\mathbb{E}[\tau(G_t)]$ decays as~$\log(t)/t$. To the best of our knowledge the case for negative~$\delta$ remains open.

Another way of increasing the number of triangles was proposed by P. Holme and B. Kim in \cite{HK02}. Similar to the model investigated here, their model alternates between a \textit{PA} step and a so-called \textit{triad formation step}, which adds a new triangle whenever it is taken. Although, this \textit{triad formation} mechanism generates graphs with large amount of triangles, by itself it is not enough to achieve positive global clustering. In \cite{ORS18} the authors proved that~$\tau(G_t)$ is of order~$1/\log(t)$, \textit{w.h.p}. 

Positive global clustering is one of the features E. Jacob and P. M\"{o}rters desired to achieve in \cite{JM15}. Their model combines preferential attachment rule with spatial proximity. A new vertex $v$ is placed uniformly on the one-dimensional torus and the rule (\ref{eq:PArule}) becomes a function of the degree of $u$ and the distance between $v$ and $u$. This mechanism exhibits two regimes for the global clustering: one that $\tau(G_t)$ converges in probability to a positive constant and one where it converges to zero. In our case, we prove a stronger result than the expected value of $\tau(G_t)$. We prove a concentration inequality result for $\tau(G_t)$, which is stated below.
\begin{theorem}[Global clustering coefficient]\label{thm:clustering} For any positive $\varepsilon < 1$, there exist positive constants $C_1$ and~$C_2$, depending on $\varepsilon$ and $p$ only, such that
	\[
	\Pd \left( \frac{C_1}{t^{\gamma(p)(1+\varepsilon)}} \le \tau(G_t) \le \frac{C_2}{t^{\gamma(p)(1-\varepsilon)}}\right) = 1- o(1),
	\]
	where $\gamma$ is the positive function:
	\[
	\gamma(p) := 2-p - \frac{3(1-p)}{2-p}.
	\]
\end{theorem}
Regarding the edge-step dynamic, one may think that the more edge-steps we take, the more clustered the graph is. As the above theorem states, the $\tau(G_t)$ is largest when $\gamma(p)\in[0,1]$ is at its minimum. Turns out that this minimum is not achieved in $p=0$, case in which we only perform edge-steps. In fact, Theorem~\ref{thm:clustering} shows that this model presents its highest global clustering when~$p = 2 - \sqrt{3}$.
\subsection{Clique number} 
The clique number of a graph $G$ (denoted by $\omega(G)$) is defined as the number of vertices in the largest complete subgraph (clique) in $G$. Regarding the existence of cliques, in \cite{ARS17}, the authors proved that for any $\varepsilon$ the graph~$G_t$ has \textit{w.h.p} a clique of order $t^{(1-\varepsilon)(1-p)/(2-p)}$. In this paper we prove that, up to sub-polynomially small factors, this is the order of largest clique in $G_t$. More precisely we prove the following theorem
\begin{theorem}[The clique number]\label{teo:clique} For any positive $\varepsilon <1$
	\[
	\Pd \left( t^{\frac{(1-\e)(1-p)}{2-p}} \le \omega(G_t) \le t^{\frac{(1-p)}{2-p}}\log^3(t)\right) = 1 - o(1).
	\]
\end{theorem}
The above theorem illustrates that the \textit{edge-step}, even when taken in much smaller proportion than the \textit{vertex-step}, is capable of producing robust substructures on the graphs that are not observed on the traditional BA-model and many other modifications of it. 
\subsection{Organization} In Section \ref{sec:boundsfordeg} we establish the machinery behind our main results proving useful estimates for the vertices' degree. In Section~\ref{sec:cherries} we apply the bounds obtained in the previous section to bound the number of paths of length $2$ at time $t$, which is the denominator of $\tau(G_t)$. In Section \ref{sec:exptriangles} we prove upper bounds for the number of triangles in~$G_t$. Finally, for the sake of organization of the paper, in Section \ref{sec:mainresults} we just combine the results proven in the previous sections to prove theorems \ref{thm:clustering} and \ref{teo:clique}, respectively.

\section{Bounds for the degree}\label{sec:boundsfordeg}
This section is devoted to obtaining sharp upper bounds for the vertices' degrees and to guarantee the existence of at least one vertex with very high degree. These estimates will be needed to derive an upper for the number of triangles in $G_t$ and to bound the number of cherries as well.

Since the number of vertices is random, we use the letters $i,j, k$, mostly, to express the $i$\textit{-th vertex added by the process}. In this way, $i$ will be used as \textit{an integer number} and as a vertex itself. We also let $d_t(i)$ to be the degree of $i$-th vertex at time $t$.
\subsection{Lower bound for the degree}
In this part our aim is to assure the existence of a vertex with very high degree. For this we evoke Theorem 2 of \cite{ARS17} setting there, for a fixed $\varepsilon >0$, $j=t^{\varepsilon}$ and $m$ large enough so $1-\delta_m = 1-\varepsilon$. In these settings the mentioned theorem gives us the following corollary 
\begin{corollary}[of Theorem 2 in \cite{ARS17} ] \label{cor:lowerdeg}Given $\varepsilon >0$, there exist positive constants $C_1,C_2$ and~$\delta$ depending on $\varepsilon$ and $p$ only such that
\[
\Pd \left( \exists j \in G_t, d_t(j) \ge C_1t^{c_p(1-\varepsilon)}\right) \ge 1 - C_2t^{-\delta}.
\]
\end{corollary}
It will be useful for us to estimate how many neighbors a vertex whose degree is at least~$C_1t^{c_p(1-\varepsilon)}$ has. For this, we prove the lemma below which is essentially the statement of the above corollary for numbers of neighbors, $\Gamma_t(j)$.
\begin{lemma}\label{lemma:neigh} Given $\varepsilon >0$, there exist positive constants $C'_1, C'_2$ and $\delta'$ depending on $\varepsilon$ and $p$ only such that
	\[
	\Pd \left( \exists j \in G_t, \Gamma_t(j) \ge C'_1t^{c_p(1-\varepsilon)}\right) \ge 1 - C'_2t^{-\delta'}.
	\]
\end{lemma}
\begin{proof} By Corollary \ref{cor:lowerdeg} with probability at least $1 - C_2t^{-\delta}$ there exists in $G_t$ a vertex $j$ with degree at least $ C_1t^{c_p(1-\varepsilon)}$. We claim that the number of neighbors of $j$ at time~$2t$ that connect to~$j$ between times $t$ and $2t$ is at least $C'_1t^{c_p(1-\varepsilon)}$, \textit{w.h.p}. To see this, let $\zeta_s$ be the following random variable
	\begin{equation}
	\zeta_s = \mathbb{1}\left \lbrace \text{a vertex is added at step }s\text{ and it connects to }j \right \rbrace.
	\end{equation}
Observe that for $s\in [t+1,2t]$ we have
\begin{align*}
\Ed \left[ \zeta_{s+1} \middle | G_{s}, d_t(j) \ge C_1t^{c_p(1-\varepsilon)}\right] 
&=
\Ed \left[p\frac{d_s(j)}{2s} \middle | G_{s}, d_t(j) \ge C_1t^{c_p(1-\varepsilon)}\right] 
\\
&\ge \frac{C_1p}{4t^{\varepsilon + 2^{-1}p(1-\varepsilon)}}.
\end{align*}
Thus, the random variable $N$ defined as
\begin{equation}
	N := \sum_{s=t+1}^{2t} \zeta_s,
\end{equation}
which counts the number of neighbors $j$ has gained, between time $t$ and $2t$, only by \textit{vertex-steps} conditioned on $j$ having degree large enough dominates a binomial random variable with parameter~$t$ and $C_1 4^{-1}pt^{-\varepsilon - p2^{-1}(1-\varepsilon)}$ which, by its turn, is exponentially concentrated around its mean~$C_14^{-1}pt^{c_p(1-\varepsilon)}$. This proves the lemma. 
\end{proof}
\subsection{Upper bound for the degree} In this part we obtain a sharp upper bound for the degree of a fixed vertex $i$. Since the proof relies on the fact that the degree of a vertex properly normalized is a martingale, we define below this normalizing factor:
	\begin{equation}\label{def:phi}
	\phi(t) := \prod_{s=1}^{t-1} \left(1+\frac{c_p}{s} \right),
	\end{equation}
where $c_p$ is a function of $p$ defined as
\begin{equation}\label{def:cp}
c_p := 1 - \frac{p}{2}.
\end{equation}
A useful fact about $\phi$ is that there exists $c_1,c_2 >0$ such that $c_2t^{c_p} \le \phi_1(t) \le c_1t^{c_p}$ for all~$t$. We will need this multiple times throughout the paper.
Now we go to the proof of the main result of this section.
\begin{theorem}[Upper bound for the degree]\label{thm:upper}\label{teo:uppbounddeg} There exist positive constants $C_1$ and $C_2$, depending on $p$ only, such that for every vertex $i$ and every number $\lambda>i^{-c_p}$ the following upper bound holds
	\begin{equation}
	\label{eq:uppbounddeg}
	\Pd\left( \sup_{s \in \N}\left \lbrace \frac{d_s(i)}{\phi(s)} \right \rbrace > \frac{\lambda}{i^{c_p}} \right) \le C_1\exp\{-C_2\lambda\}.
	\end{equation}
\end{theorem}
\begin{proof} To simplify our writing, we let~$\mathbb{P}_{G_{t_i}}$ be the probability measure $\Pd$ conditioned on the event where the graph~$G_t$ is some given graph~$G_{t_i}$ to which the $i$-th vertex has just been added. By Proposition 2.1 of~\cite{ARS17}, the sequence~$\{X_{s,t_i}\}_{s \ge 1}$ defined below as
	\begin{equation}\label{def:xs}
		X_{s,t_i} := \frac{d_s(i)}{\phi(s)}
	\end{equation}
is a martingale of mean $\phi^{-1}(t_i)$ with respect the natural filtration $\{\mathcal{F}_s\}_{s\ge 1}$ and the measure~$\Pd_{G_{t_i}}$. 
 In this setting, for a fixed positive number $\lambda$, let $\eta$ be the stopping time
\begin{equation}\label{def:eta}
	\eta := \inf \left \lbrace s \ge 1 ; X_{s,t_i} \ge \lambda \right \rbrace.
\end{equation}
Then, we define the following stopped martingale:
\begin{equation}
	X'_s := X_{s \wedge \eta,t_i}.
\end{equation}
Observe that the increment ($\Delta X_s := X'_{s+1} - X'_s$) of the stopped martingale satisfies, for $s\ge t_i$,
\begin{equation}\label{eq:incrx}
\begin{split}
\left | \Delta X'_s\right | = \left | \frac{d_{s+1}(i)}{\phi(s+1)} - \frac{d_{s}(i)}{\phi(s)}\right |\mathbb{1}_{\{\eta > s \}} = \left | \frac{\Delta d_{s+1}(i)}{\phi(s+1)} - \frac{c_p d_{s}(i)}{s\phi(s+1)}\right |\mathbb{1}_{\{\eta > s \}} \le \frac{C}{\phi(s+1)},
\end{split}
\end{equation}
since $d_s(i) \le 2s$ for all $s$ deterministically and
\begin{equation}
\Delta d_s(i) \le 2\mathbb{1}\{i\text{ is chosen at least once at step }s+1\}.
\end{equation}
Combining the above bound with the the second identity on (\ref{eq:incrx}), we also obtain, for~$s>t_i$,
\begin{equation}
\begin{split}
\Ed_{G_{t_i}} \left[  \left(\Delta X'_s\right)^2\middle | \mathcal{F}_{s} \right] & \le 2 \Ed_{G_{t_i}} \left[  \frac{(\Delta d_{s+1}(i))^2}{\phi^2(s+1)}\middle | \mathcal{F}_{s} \right]\mathbb{1}_{\{\eta > s \}} + \frac{2c^2_p d^2_{s}(i)}{s\phi^2(s+1)}\mathbb{1}_{\{\eta > s \}} \\
&\le
\frac{2}{\phi^2(s+1)} \Ed_{G_{t_i}} \left[ 4\cdot\mathbb{1}^2\{i\text{ is chosen at least once at step }s+1\}\middle | \mathcal{F}_{s} \right]\mathbb{1}_{\{\eta > s \}}\\
&\quad + \frac{2c^2_p d^2_{s}(i)}{s\phi^2(s+1)}\mathbb{1}_{\{\eta > s \}}
\\
& \le \left( \frac{Cd_s(i)}{s\phi^2(s+1)} + \frac{2c^2_p d^2_{s}(i)}{s^2\phi^2(s+1)} \right)\mathbb{1}_{\{\eta > s \}} \\
& \le \frac{C\lambda }{s\phi(s+1)} + \frac{4c^2_p\lambda}{s\phi(s+1)} .
\end{split}
\end{equation}
The above inequality implies that
\begin{equation}
W'_t := \sum_{s=1}^{(t-1)\wedge \eta }\Ed_{G_{t_i}} \left[  \left(\Delta X'_s\right)^2\middle | \mathcal{F}_{s} \right] \le \sum_{s=t_i}^{t-1} \frac{C \lambda }{s\phi(s+1)} \le \frac{C\lambda}{t^{c_p}_i}, \text{a. s.}
\end{equation}
Now we use Freedman's inequality \cite{F75} (or Theorem 6 in \cite{ORS18} for a more concise statement) to obtain that for any positive constant $A$
\begin{equation}\label{eq:freed}
\Pd_{G_{t_i}} \left( X'_t - \phi^{-1}(t_i) \ge A \right) \le \exp \left \lbrace - \frac{A^2}{\frac{2C\lambda}{t^{c_p}_i}+\frac{2\cdot 2\cdot A}{3t_i^{c_p}}}\right \rbrace.
\end{equation}
Now, we would like to guarantee that the stopping time $\eta$ is not too small, i.e., that the martingales $X'$ and $X$ are essentially the same. To do this observe that
\begin{equation}\label{eq:eta}
\begin{split}
	\Pd_{G_{t_i}} \left( \eta \le t \right) \le 	\Pd_{G_{t_i}} \left( \exists s \le t, X_s \ge \lambda \right) = 	 \Pd_{G_{t_i}} \left( X'_t  - \phi^{-1}(t_i) \ge  \lambda - \phi^{-1}(t_i) \right)
\end{split}
\end{equation}
and taking $A = \lambda - \phi^{-1}(t_i)$, which is positive by the hypothesis on $\lambda$, in (\ref{eq:freed}) we obtain
\begin{equation}
\begin{split}
\Pd_{G_{t_i}} \left( X'_t  - \phi^{-1}(t_i) \ge  \lambda - \phi^{-1}(t_i) \right) & \le \exp \left \lbrace - \frac{(\lambda - \phi^{-1}(t_i) )^2}{\frac{2C\lambda}{t^{c_p}_i}+\frac{4(\lambda - \phi^{-1}(t_i) )}{3t_i^{c_p}}}\right \rbrace 
 \le e^{C'}\exp \left \lbrace  -C\lambda t_i^{c_p}\right \rbrace,
\end{split}
\end{equation}
Combining the above bound with (\ref{eq:eta}) we obtain
\begin{equation}
\Pd_{G_{t_i}} \left(\eta < \infty \right) = \lim_{t \to \infty } \Pd_{G_{t_i}} \left(\eta \le t \right) \le C'\exp \left \lbrace  -C\lambda t_i^{c_p}\right \rbrace.
\end{equation}
Replacing $\lambda$ by $\lambda/i^{c_p}$ on the definition of $\eta$ in (\ref{def:eta}) we finally obtain
\begin{equation}
\Pd_{G_{t_i}} \left( \sup_{s \in \N} \left \lbrace \frac{d_s(i)}{\phi(s)}\right \rbrace \ge \frac{\lambda}{i^{c_p}}\right) = \Pd_{G_{t_i}} \left(\eta < \infty \right) \le C'\exp \left \lbrace  -\frac{C\lambda t_i^{c_p}}{i^{c_p}}\right \rbrace
\end{equation}
and since $t_i \ge i$, integrating on $G_{t_i}$ yields
\begin{equation}
\Pd \left( \sup_{s \in \N} \left \lbrace \frac{d_s(i)}{\phi(s)}\right \rbrace \ge \frac{\lambda}{i^{c_p}}\right) \le C'\exp \left \lbrace  -C\lambda\right \rbrace,
\end{equation}
which proves the theorem.
\end{proof}
\section{Concentration results for the number of cherries}\label{sec:cherries}
In this section we combine the bounds obtained in the previous section to prove concentration inequalities for $\mathcal{C}(G_t)$, i.e., the number of paths of length two or simply \textit{cherries}. Our aim is to prove the theorem below
\begin{theorem}[Concentration for cherries]\label{teo:cerejas} Given $\varepsilon$, there exist positive constants $C_1, C_2, C_3$ and $\delta$, depending on $\varepsilon$ and $p$ only, such that 
\[
\Pd \left( C_1t^{(2-p)(1-\varepsilon)} \le \mathcal{C}(G_t) \le C_2t^{(2-p)}\log^2t \right) \ge 1- C_3t^{-\delta}.
\]
\end{theorem}
\begin{proof} The lower bound follows immediately from Lemma~\ref{lemma:neigh}. Observe that from that lemma we have with probability at least $1-C_1't^{-\delta}$ a vertex $j$ having at least $C_2't^{(1-p/2)(1-\varepsilon)}$ neighbors. So, the cherries coming from $j$ are already of order $t^{(2-p)(1-\varepsilon)}$.
	On the other hand, observe that
	\begin{equation}\label{eq:upct}
	\mathcal{C}(G_t) \le \sum_{v \in G_t} \binom{d_t(v)}{2}.
	\end{equation}
We use the following definition:
\begin{definition}
	\label{def:Cp}
	Let~$C_p>0$ be such that the right hand side of~\eqref{eq:uppbounddeg} is smaller than~$t^{-10}$ if~$\lambda$ is chosen as~$C_p \log t$.
\end{definition} 
Then, using Theorem~$\ref{teo:uppbounddeg}$ and a union bound, we obtain
\begin{equation}
\Pd \left( \bigcup_{i \in G_t} \left \lbrace d_t(i) \ge C_p\frac{t^{c_p}\log(t)}{i^{c_p}}\right \rbrace \right) \le C't^{-9}.
\end{equation}
Thus, with probability at least $1-t^{-9}$ we have
\begin{equation}
\begin{split}
\mathcal{C}(G_t) \le \sum_{v \in G_t} \binom{d_t(v)}{2} \le C_p\sum_{i=1}^t\frac{t^{2-p}\log^2(t)}{i^{2-p}} \le Ct^{2-p}\log^2(t),
\end{split}
\end{equation}
this completes the proof.
\end{proof}

\section{The expected number of triangles}\label{sec:exptriangles}

In this section we prove an upper bound for the expected number of triangles (counted without multiplicities), denoted by $\mathcal{T}(G_t)$, in $G_t$.
If we let $\mathrm{edg}_t(i,j)$ be the integer \textit{r.v.} which counts the \textit{number of edges} connecting vertices $i$ and $j$ at time $t$, the number of triangles in $G_t$ may be written as
\begin{equation}
	\mathcal{T}(G_t) = \sum_{1\le i<j<k\le t} \mathbb{1}\{ \mathrm{edg}_t(i,j) \mathrm{edg}_t(i,k)\mathrm{edg}_t(j,k) \ge 1 \}.
\end{equation}
By the above identity, estimating the expected value of $\mathcal{T}(G_t)$ is the same that estimating the probability of the product $\mathrm{edg}_t(i,j) \mathrm{edg}_t(i,k)\mathrm{edg}_t(j,k)$ be at least $1$, which in turns is bounded from above by the expected value of the same product of \textit{r.v}'s. Turns out that, for suitable choices of $i,j$ and $k$, this bound is good enough, as we will see latter.

From the perspective of the above observation, this section is essentially devoted to bound the expected value of a product of correlated \textit{r.v}'s. To overcome the issue of correlation, we construct, for each triple of vertices $i,j$ and $k$, negatively correlated random variables which dominate the random variables~$\mathrm{edg}_t(i,j),  \mathrm{edg}_t(i,k)$ and $\mathrm{edg}_t(j,k)$. The result we would like to prove is stated in the proposition below.
\begin{proposition}
	\label{prop:exptriangle}
Using the notation above defined, there exists a positive constant $C$, depending on $p$ only, such that
\[
\Ed[\mathcal{T}(G_t)]\leq  C t^{3\alpha}(\log t)^8,
\]
where $\alpha$ is the function of $p$ defined as
\begin{equation}\label{def:alpha}
\alpha := \frac{1-p}{2-p}.
\end{equation}
\end{proposition}
\begin{proof}
Throughout this proof we assume the vertices~$i,j,k$ to be such that~$0<i<j<k\leq t$. For each~$s\in\N$ we consider the variables
\begin{align}
\label{eq:edgeind}
g^{i,j}_s&:=\{\text{an edge is added between }i\text{ and }j\text{ at time } s \text{ by an edge-step }  \},
\\ \nonumber
e^{i,j}_s&:=\left\{
\text{an edge is added between }i\text{ and }j\text{ at time } s \text{ by a vertex-step}
\right\},
\end{align}
and analogously define variables for the pair~$(j,k)$ and~$(i,k)$. As said before, our first goal in this section is, in a sense, to control the covariance between these random variables for suitable values of~$i,j,k,s$. Observe that, given our knowledge of $G_s$, the random variable $g_{s+1}^{i,j}$ is one if we perform an \textit{edge-step} on $G_s$ and choose $i$ and $j$ to be the tips of the new edge added. From this we deduce that
\begin{equation}
\Ed \left[ g_{s+1}^{i,j} \middle | \mathcal{F}_s\right] = (1-p)\frac{d_s(i)d_s(j)}{2s^2}.
\end{equation}
Arguing similarly for $e_{s+1}^{i,j}$, we also obtain that
\begin{equation}
\Ed \left[ e_{s+1}^{i,j} \middle | \mathcal{F}_s\right] = p\frac{d_s(i)}{2s},
\end{equation}
since we have to add $j$ at time $s$ and choose $i$ to connect it to. From the two above identities we see that we may control the probability of two vertices being connected at a fixed time by upper bounds on their degrees. 
Let~$(U_s)_{s\geq 0}$ be an i.i.d.\ sequence of random variables with uniform distribution over~$[0,1]$, this sequence being independent from the whole process~$(G_t)_{t\geq 1}$. Let~$T_n$ denote the \textit{time when the~$n$-th vertex is added to the graph}. We also denote
\begin{equation}
\label{eq:pijdef}
p^{i,j}_s :=  C_p^2\frac{\log^2(t)}{i^{c_p}j^{c_p}s^{p}} \wedge 1,\quad p^{i,k}_s :=  C_p^2\frac{\log^2(t)}{i^{c_p}k^{c_p}s^{p}}\wedge 1, \quad p^{j,k}_s :=  C_p^2\frac{\log^2(t)}{j^{c_p}k^{c_p}s^{p}}\wedge 1,
\end{equation}
\begin{equation}
\label{eq:qijdef}
q^{i,j}_s :=  C_p\frac{(\log t) \Pd(T_j=s)}{i^{c_p}s^{\frac{p}{2}}}\wedge 1,\,
q^{i,k}_s :=  C_p\frac{(\log t) \Pd(T_k=s)}{i^{c_p}s^{\frac{p}{2}}}\wedge 1, \, 
q^{j,k}_s :=  C_p\frac{(\log t) \Pd(T_k=s)}{j^{c_p}s^{\frac{p}{2}}}\wedge 1,
\end{equation}
where~$C_p$ is the constant from Definition~$\ref{def:Cp}$. For each~$s\in\N$, we now define the random variables
\[
h^{i,j}_s:=g^{i,j}_s + (1-g^{i,j}_s)\frac{\mathbb{1}\{U_s \leq p^{i,j}_s - \Ed[ g^{i,j}_s \vert \mathcal{F}_{s-1} ] \}}{1- \Ed[ g^{i,j}_s \vert \mathcal{F}_{s-1} ]   }.
\]
Note that the above variable is well-defined, since~$ \Ed[ g^{i,j}_s \vert \mathcal{F}_{s-1} ]<1-p< 1$. We analogously define~$h^{j,k}_s$ and~$h^{i,k}_s$. Notice that if~$\Ed[ g^{i,j}_s \vert \mathcal{F}_{s-1} ]$ were smaller than~$p^{i,j}_s$ almost surely, then the expectation of~$h^{i,j}_s$ would actually be~$p^{i,j}_s$ (see~\eqref{eq:hexpec} below). We will see that this is not the case, though it is true asymptotically almost surely. By constructing~$h^{i,j}_s$ we are actually giving an ``extra chance'' of success on top of the Bernoulli variable~$g^{i,j}_s$. We do so by completing the missing probability using an independent randomness source, the intended effect is to simplify the dependence between the variables. With a similar goal in mind, we define the random variables
\[
f^{i,j}_s:=e^{i,j}_s + (1-e^{i,j}_s)\frac{\mathbb{1}\{U_s \leq q^{i,j}_s - \Ed[ e^{i,j}_s \vert \mathcal{F}_{s-1} ] \}}{1- \Ed[ e^{i,j}_s \vert \mathcal{F}_{s-1} ]   }.
\]
We also analogously define~$f^{j,k}_s$ and~$f^{i,k}_s$. 

We let
\[
\eta_{i}:=\inf_{s \in \N}\left\{d_s(i)    \geq C_p \log (t) \frac{s^{c_p}}{i^{c_p}}     \right\},
\]
again analogously defining~$\eta_j$ and~$\eta_k$. We also let
\[
\tilde{\eta}:= \eta_i \wedge \eta_j \wedge \eta_k,
\]
which is the first time the degree of one of the vertices $i,j$ or $k$ has unexpectedly increased too much.
We then note that, for~$s\leq t$,
\begin{align}
\nonumber
\lefteqn{\Ed\big[h^{i,j}_s \mathbb{1}\{\tilde{\eta}>s-1  \}  \vert \mathcal{F}_{s-1} \big]}\quad
\\
\label{eq:hexpec}
&=
\mathbb{1}\{\tilde{\eta}>s-1  \}\left(         
\Ed[ g^{i,j}_s \vert \mathcal{F}_{s-1} ]
+
\big(1-\Ed[ g^{i,j}_s \vert \mathcal{F}_{s-1} ] \big) \frac{\Ed\big[\mathbb{1}\{U_s \leq p^{i,j}_s - \Ed[ g^{i,j}_s \vert \mathcal{F}_{s-1} ]\}\big]}{1- \Ed[ g^{i,j}_s \vert \mathcal{F}_{s-1} ]   }
\right)
\\
\nonumber
&=
\mathbb{1}\{ \tilde{\eta}>s-1  \}\left(         
2(1-p)\frac{ d_{s-1}(i)d_{s-1}(j)}{4(s-1)^2}
+
\left(
p^{i,j}_s -2(1-p)\frac{ d_{s-1}(i)d_{s-1}(j)}{4(s-1)^2}
\right) 
\vee 0
\right)
\\
\nonumber
&=
\mathbb{1}\{ \tilde{\eta}>s-1  \}
p^{i,j}_s,
\end{align}
and we can obtain analogous identities for~$h^{j,k}_s$ and~$h^{i,k}_s$. Now, if~$s_1,s_2,s_3,t\in\N$ are such that~$s_1<s_2<s_3<t$, we obtain
\begin{align}
\Ed\big[h^{i,j}_{s_1}h^{j,k}_{s_2}h^{i,k}_{s_3} \mathbb{1}\{\tilde{\eta}>t  \}  \big]
\label{eq:hprod}
&=
\Ed\left[  \Ed\big[h^{i,j}_{s_1}h^{j,k}_{s_2}h^{i,k}_{s_3}\mathbb{1}\{\tilde{\eta}>t  \}  \vert \mathcal{F}_{s_3-1} \big]\right]
\\
\nonumber
&\leq
\Ed\left[  \mathbb{1}\{\tilde{\eta}>s_3-1  \} h^{i,j}_{s_1}h^{j,k}_{s_2}\Ed\big[h^{i,k}_{s_3}  \vert \mathcal{F}_{s_3-1} \big]\right]
\\
\nonumber
&\leq
p^{i,k}_{s_3}\Ed\left[  \mathbb{1}\{\tilde{\eta}>s_2 -1  \} h^{i,j}_{s_1}h^{j,k}_{s_2}\right]
\\
\nonumber
&\leq
p^{i,j}_{s_1}p^{j,k}_{s_2}p^{i,k}_{s_3},
\end{align}
and the same can be proved analogously for any of the~$5$ other different orderings of~$s_1,s_2,s_3$. Similarly, we also obtain
\begin{align}
\label{eq:fexpec}
\Ed\big[f^{i,j}_s \mathbb{1}\{\tilde{\eta}>s-1  \}  \vert \mathcal{F}_{s-1} \big]
=
\mathbb{1}\{ \tilde{\eta}>s-1  \}
q^{i,j}_s,
\end{align}
and
\begin{align}
\Ed\big[f^{i,j}_{s_1}h^{j,k}_{s_2}h^{i,k}_{s_3} \mathbb{1}\{\tilde{\eta}>t  \}  \big]
\label{eq:hfprod}
&\leq
q^{i,j}_{s_1}p^{j,k}_{s_2}p^{i,k}_{s_3},\quad
\Ed\big[h^{i,j}_{s_1}f^{j,k}_{s_2}f^{i,k}_{s_3} \mathbb{1}\{\tilde{\eta}>t  \}  \big]
\leq
p^{i,j}_{s_1}q^{j,k}_{s_2}q^{i,k}_{s_3},
\\ \nonumber
\Ed\big[h^{i,j}_{s_1}f^{j,k}_{s_2}h^{i,k}_{s_3} \mathbb{1}\{\tilde{\eta}>t  \}  \big]
&\leq
p^{i,j}_{s_1}q^{j,k}_{s_2}p^{i,k}_{s_3},\quad
\Ed\big[f^{i,j}_{s_1}h^{j,k}_{s_2}f^{i,k}_{s_3}  \mathbb{1}\{\tilde{\eta}>t  \}  \big]
\leq
q^{i,j}_{s_1}p^{j,k}_{s_2}q^{i,k}_{s_3},
\\ \nonumber
\Ed\big[h^{i,j}_{s_1}h^{j,k}_{s_2}f^{i,k}_{s_3}  \mathbb{1}\{\tilde{\eta}>t  \}  \big]
&\leq
p^{i,j}_{s_1}p^{j,k}_{s_2}q^{i,k}_{s_3},\quad
\Ed\big[f^{i,j}_{s_1}f^{j,k}_{s_2}h^{i,k}_{s_3}  \mathbb{1}\{\tilde{\eta}>t  \}  \big]
\leq
q^{i,j}_{s_1}q^{j,k}_{s_2}p^{i,k}_{s_3},
\end{align}

We note that if~$s_1=s_2$ for example, then
\[
\Ed[e^{i,j}_{s_1}e^{j,k}_{s_1}g^{i,k}_{s_3}]=\Ed[e^{i,j}_{s_1}g^{j,k}_{s_1}g^{i,k}_{s_3}]=\Ed[g^{i,j}_{s_1}g^{j,k}_{s_1}g^{i,k}_{s_3}]=0,
\]
since it is not possible to put edges between different pairs of vertices at the same time step, and that the expectation of any other triple product of these types of variables is~$0$ if any of the times steps $s_1,s_2,s_3$ coincide. We then obtain, for any triplet of numbers~$s_1,s_2,s_3\leq t$,
\begin{align}
\label{eq:ghdom}
\Ed\left[g_{s_1}^{i,j}g_{s_2}^{i,k} g_{s_3}^{j,k}          \right] 
&\leq 
\Ed\left[ g_{s_1}^{i,j} g_{s_2}^{i,k} g_{s_3}^{j,k} \mathbb{1}\{ \tilde \eta> t    \}             \right]
+
\Pd (\tilde \eta\leq t )
\\
\nonumber
&\leq 
p^{i,j}_{s_1}p^{j,k}_{s_2}p^{i,k}_{s_3}
+
\Pd (\tilde \eta\leq t )
\\
\nonumber
&\leq
(C_p \log (t))^6 \frac{1}{(ijk)^{2-p}(s_2 s_2 s_3)^p}
+
\Pd (\tilde \eta\leq t ).
\end{align}
Again, we similarly obtain
\begin{align}
\label{eq:ghefdom}
\Ed\left[e_{s_1}^{i,j}g_{s_2}^{i,k} g_{s_3}^{j,k}          \right] 
&\leq
C_p^5(\log t)^5 \frac{1}{(ik)^{2-p}j^{c_p}}\frac{1}{(s_2 s_3)^{p}}\frac{\Pd(T_j=s_1)}{s_1^{p/2}}
+
\Pd (\tilde \eta\leq t ),
\\ \nonumber
\Ed\left[g_{s_1}^{i,j}e_{s_2}^{i,k} g_{s_3}^{j,k}          \right] 
&\leq
C_p^5 (\log t)^5 \frac{1}{(ij)^{2-p}k^{c_p}}\frac{1}{(s_1 s_3)^{p}}\frac{\Pd(T_k=s_2)}{s_2^{p/2}}
+
\Pd (\tilde \eta\leq t ),
\\ \nonumber
\Ed\left[g_{s_1}^{i,j}g_{s_2}^{i,k} e_{s_3}^{j,k}          \right] 
&\leq
C_p^5 (\log t)^5 \frac{1}{(ij)^{2-p}k^{c_p}}\frac{1}{(s_1 s_2)^{p}}\frac{\Pd(T_k=s_3)}{s_3^{p/2}}
+
\Pd (\tilde \eta\leq t ),
\\ \nonumber
\Ed\left[e_{s_1}^{i,j}g_{s_2}^{i,k} e_{s_3}^{j,k}          \right] 
&\leq
C_p^4 (\log t)^4 \frac{1}{(ik)^{c_p}s_2^{p}}\frac{\Pd(T_j=s_1)}{i^{c_p} s_1^{p/2}}\frac{\Pd(T_k=s_3)}{j^{c_p} s_3^{p/2}}
+
\Pd (\tilde \eta\leq t ),
\\ \nonumber
\Ed\left[e_{s_1}^{i,j}e_{s_2}^{i,k} g_{s_3}^{j,k}          \right] 
&\leq
C_p^4 (\log t)^4 \frac{1}{(jk)^{c_p}s_3^{p}}\frac{\Pd(T_j=s_1)}{i^{c_p} s_1^{p/2}}\frac{\Pd(T_k=s_2)}{i^{c_p} s_2^{p/2}}
+
\Pd (\tilde \eta\leq t ),
\end{align}
 We observe that, since $i<j<k$, it is impossible for~$k$ to be connected to both~$i$ and~$j$ via vertex-steps, and therefore~$\Ed[g^{i,j}_{s_1}e^{j,k}_{s_2}e^{i,k}_{s_3}]=0$.

Now we use the above inequalities in order to bound~$\Ed[\mathcal{T}(G_t)]$ from above. We recall that~$\mathcal{T}(G_t)$ counts the number of triangles in~$G_t$ disregarding the multiplicity of edges. Therefore, in order for the above discussion to be useful, it will be important to estimate the number of triangles formed by earlier vertices (which usually have high degree, which corresponds to a high multiplicity of edges) and triangles formed by later vertices separately. We let
\begin{align}
\label{eq:t1t2def}
\mathcal{T}_1(G_t)&:=\#\big\{\{i,j,k\}\subset\N ; i\cdot j \cdot k \leq t^{3\alpha}   \big\},
\\
\nonumber
\mathcal{T}_2(G_t)&:=\#
\left\{
\begin{array}{c}
\{i,j,k\}\subset\N ; i\cdot j \cdot k \geq t^{3\alpha};i<j<k;
\\ \text{ and the vertices } i,j,k \text{ form a triangle in }G_t    
\end{array}
\right\}.
\end{align}
We then have
\begin{equation*}
\mathcal{T}(G_t)\leq \mathcal{T}_1(G_t) + \mathcal{T}_2(G_t).
\end{equation*}
Now, $\mathcal{T}_1(G_t)$ can be estimated in an elementary way:
\begin{equation}\label{eq:t1bound}
\begin{split}
\mathcal{T}_1(G_t)
&\leq
\sum_{i=1}^{t^{3\alpha}}\sum_{j=1}^{\frac{t^{3\alpha}}{i}}\sum_{k=1}^{\frac{t^{3\alpha}}{ij}} 1 \leq
Ct^{3\alpha}(\log t)^2.
\end{split}
\end{equation}
But~$\mathcal{T}_2(G_t)$ is more complicated. We have to break it into three distinct sets:
\begin{align}
\label{eq:t2subsets}
\mathcal{T}_2^0(G_t)&:=\#
\left\{
\begin{array}{c}
\{i,j,k\}\subset\N ; i\cdot j \cdot k \geq t^{3\alpha};i<j<k;
\\ \text{ and the vertices } i,j,k \text{ form a triangle in }
\\ G_t    \text{ with all edges coming from edge-steps}
\end{array}
\right\},
\\ \nonumber
\mathcal{T}_2^1(G_t)&:=\#
\left\{
\begin{array}{c}
\{i,j,k\}\subset\N ; i\cdot j \cdot k \geq t^{3\alpha};i<j<k;
\\ \text{ and the vertices } i,j,k \text{ form a triangle in }G_t  \text{ with two edges}  
\\ \text{ coming from edge-steps and one from a vertex-step}
\end{array}
\right\},
\\ \nonumber
\mathcal{T}_2^2(G_t)&:=\#
\left\{
\begin{array}{c}
\{i,j,k\}\subset\N ; i\cdot j \cdot k \geq t^{3\alpha};i<j<k;
\\ \text{ and the vertices } i,j,k \text{ form a triangle in }G_t    \text{ with one edge}
\\ \text{ coming from an edge-step and two from vertex-steps}
\end{array}
\right\}.
\end{align}
Note that it is impossible for a triangle to be formed by three edges coming from vertex-steps. Therefore,
\[
\mathcal{T}_2(G_t) \leq \mathcal{T}_2^0(G_t)+\mathcal{T}_2^1(G_t)+\mathcal{T}_2^2(G_t).
\]
We bound the expectations of the variables in the right hand side of the above inequality separately. First, recalling that~$\alpha=\alpha(p)=(1-p)(2-p)^{-1}$ and bounding the summand by the integral, we have,
\begin{align}
\nonumber
\Ed[\mathcal{T}_2^0(G_t)]
&\leq
\Ed\left[        \sum_{i=1}^{t}\sum_{j=\frac{t^{3\alpha}}{i}}^{t}\sum_{k=\frac{t^{3\alpha}}{ij}}^{t}    \sum_{s_1=i}^{t}    \sum_{s_2=j}^{t} \sum_{s_3=k}^{t}
g^{i,j}_{s_1}g^{j,k}_{s_2}g^{i,k}_{s_3} \right]
\\ \label{eq:t20bound}
&\stackrel{\eqref{eq:ghdom}}{\leq}
\sum_{i=1}^{t}\sum_{j=\frac{t^{3\alpha}}{i}}^{t}\sum_{k=\frac{t^{3\alpha}}{ij}}^{t}   \sum_{s_1=1}^{t}    \sum_{s_2=1}^{t} \sum_{s_3=1}^{t}
\left((C_p \log (t))^6 \frac{1}{(ijk)^{2-p}(s_2 s_2 s_3)^p}
+
\Pd (\tilde \eta\leq t ) \right)
\\
\nonumber
&\stackrel{\eqref{eq:uppbounddeg}}{\leq}
C t^{3(1-p)}(\log t)^6\left(
\sum_{i= 1}^{t}\sum_{j=\frac{t^{3\alpha}}{i}}^{t}\sum_{k\geq\frac{t^{3\alpha}}{ij}}
\frac{1}{(ijk)^{2-p}}
+3t^{-10}   \right)
\\
\nonumber
&\leq
C(\log t)^8  t^{3(1-p)(1-\alpha)}
\\
\nonumber
&=
C(\log t)^8  t^{3\alpha}.
\end{align}
By an elementary Bernstein bound, one can prove that the variable~$T_n$ is concentrated (with one minus exponentially small probability) around~$n p^{-1}$. Therefore, we have
\begin{equation}\label{eq:exptaub}
\sum_{s=1}^{t} \frac{\Pd(T_n=s)}{s^{\frac{p}{2}}}\leq \Ed\left[T_n^{-\frac{p}{2}}\right]\leq c n^{-\frac{p}{2}}.
\end{equation}
Again, using the bounds derived in \eqref{eq:ghefdom} and the integral bound, we can then obtain
\begin{align}
\nonumber
\Ed[\mathcal{T}_2^1(G_t)]
&\leq
\Ed\left[        \sum_{i=1}^{t}\sum_{j=\frac{t^{3\alpha}}{i}}^{t}\sum_{k=\frac{t^{3\alpha}}{ij}}^{t}    \sum_{s_1=i}^{t}    \sum_{s_2=j}^{t} \sum_{s_3=k}^{t}
e^{i,j}_{s_1}g^{j,k}_{s_2}g^{i,k}_{s_3} \right]
\\ \label{eq:t21bound}
&\quad +
\Ed\left[        \sum_{i=1}^{t}\sum_{j=\frac{t^{3\alpha}}{i}}^{t}\sum_{k=\frac{t^{3\alpha}}{ij}}^{t}    \sum_{s_1=i}^{t}    \sum_{s_2=j}^{t} \sum_{s_3=k}^{t}
g^{i,j}_{s_1}e^{j,k}_{s_2}g^{i,k}_{s_3} \right]
\\ \nonumber
&\quad +
\Ed\left[        \sum_{i=1}^{t}\sum_{j=\frac{t^{3\alpha}}{i}}^{t}\sum_{k=\frac{t^{3\alpha}}{ij}}^{t}    \sum_{s_1=i}^{t}    \sum_{s_2=j}^{t} \sum_{s_3=k}^{t}
g^{i,j}_{s_1}g^{j,k}_{s_2}e^{i,k}_{s_3} \right]
\\ \nonumber
&\stackrel{\eqref{eq:exptaub}}{\leq}
C t^{2(1-p)} (\log t)^5 \left(
\sum_{i=1}^{t}\sum_{j=\frac{t^{3\alpha}}{i}}^{t}\sum_{k=\frac{t^{3\alpha}}{ij}}^{t}   \frac{2}{(ij)^{2-p}k} +\frac{1}{(ik)^{2-p}j}  \right)
+ 3t^6 \Pd(  \tilde{\eta} \leq t  )
\\ \nonumber
&\leq C t^{2(1-p)} (\log t)^5
\left(
(\log t)^2t^{-3\alpha(1-p)}   +\log t\cdot t^{1-p}\cdot  t^{-3\alpha (1-p)}
\right)
\\ \nonumber
&\leq C(\log t)^7 t^{3\alpha}.
\end{align}
We conclude by estimating the expectation of~$\mathcal{T}_2^2(G_t)$ in a similar manner as above:
\begin{align}
\nonumber
\Ed[\mathcal{T}_2^2(G_t)]
&\leq
\Ed\left[        \sum_{i=1}^{t}\sum_{j=\frac{t^{3\alpha}}{i}}^{t}\sum_{k=\frac{t^{3\alpha}}{ij}}^{t}    \sum_{s_1=i}^{t}    \sum_{s_2=j}^{t} \sum_{s_3=k}^{t}
e^{i,j}_{s_1}g^{j,k}_{s_2}e^{i,k}_{s_3} \right]
\\ \label{eq:t22bound}
&\quad +
\Ed\left[        \sum_{i=1}^{t}\sum_{j=\frac{t^{3\alpha}}{i}}^{t}\sum_{k=\frac{t^{3\alpha}}{ij}}^{t}    \sum_{s_1=i}^{t}    \sum_{s_2=j}^{t} \sum_{s_3=k}^{t}
e^{i,j}_{s_1}e^{j,k}_{s_2}g^{i,k}_{s_3} \right]
\\ \nonumber
&\leq
\sum_{i=1}^{t}\sum_{j=\frac{t^{3\alpha}}{i}}^{t}\sum_{k=\frac{t^{3\alpha}}{ij}}^{t}    \sum_{s_1=i}^{t}    \sum_{s_2=j}^{t} \sum_{s_3=k}^{t}
C_p^4 (\log t)^4 \frac{1}{(ik)^{c_p}s_2^{p}}\frac{\Pd(T_j=s_1)}{i^{c_p} s_1^{p/2}}\frac{\Pd(T_k=s_3)}{j^{c_p} s_3^{p/2}}
\\ \nonumber
&\quad +
\sum_{i=1}^{t}\sum_{j=\frac{t^{3\alpha}}{i}}^{t}\sum_{k=\frac{t^{3\alpha}}{ij}}^{t}    \sum_{s_1=i}^{t}    \sum_{s_2=j}^{t} \sum_{s_3=k}^{t}
C_p^4 (\log t)^4 \frac{1}{(jk)^{c_p}s_3^{p}}\frac{\Pd(T_j=s_1)}{i^{c_p} s_1^{p/2}}\frac{\Pd(T_k=s_2)}{i^{c_p} s_2^{p/2}}
\\ \nonumber
&\quad + 3t^6 \Pd(  \tilde{\eta} \leq t  )
\\ \nonumber
&\stackrel{\eqref{eq:exptaub}}{\leq}
C t^{1-p} (\log t)^4 \left(
\sum_{i=1}^{t}\sum_{j=\frac{t^{3\alpha}}{i}}^{t}\sum_{k=\frac{t^{3\alpha}}{ij}}^{t}   \frac{2}{i^{2-p}jk} \right)
+ 3t^6 \Pd(  \tilde{\eta} \leq t  )
\\ \nonumber
&\leq C t^{1-p} (\log t)^6,
\end{align}
which finally implies
\[
\Ed[\mathcal{T}(G_t)]\leq C( \log t)^8 t^{3\alpha},
\]
concluding the proof of Proposition~$\ref{prop:exptriangle}$.
\end{proof}
\section{Proof of the main results}\label{sec:mainresults}
In this part we wrap up all the results we have proven so far to prove our two main results: Theorem~\ref{thm:clustering} and Theorem \ref{teo:clique}.
\begin{proof}[Proof of Theorem \ref{thm:clustering} - Clustering Coefficient] First recall that $\tau(G_t)$ is defined as three times $\mathcal{T}(G_t)/\mathcal{C}(G_t)$.
Then, by Theorem \ref{teo:cerejas}, we have that, for any $\varepsilon < 1$, 
\[
 C_1t^{(2-p)(1-\varepsilon)} \le \mathcal{C}(G_t) \le C_2t^{(2-p)}\log^2t,
\]
for some constants $C_1$ and $C_2$, with probability $1-o(1)$. For $\mathcal{T}(G_t)$, we evoke Theorem 1 of \cite{ARS17}, which states that there exists, with probability~$1-o(1)$, a clique of order $t^{\alpha(1-\varepsilon)}$ in~$G_t$. Thus, with probability~$1-o(1)$,~$\mathcal{T}(G_t)$ is at least $t^{3\alpha(1-\varepsilon)}$. And finally, by Proposition~\ref{prop:exptriangle} and Markov's inequality, we have that~$\mathcal{T}(G_t)$ is at most $t^{3\alpha}\log^9t$, with probability~$1-o(1)$. This proves the theorem.

\end{proof}
\begin{proof}[Proof of Theorem \ref{teo:clique} - Clique Number] The existence of a clique of order $t^{\frac{(1-\e)(1-p)}{2-p}}$ in $G_t$, \textit{w.h.p}, was proved by \cite{ARS17} in Theorem 1. For the upper bound, observe that the existence of a complete subgraph of order~$C^{1/3}t^{\frac{(1-p)}{2-p}}\log^3(t)$ in $G_t$ implies immediately that $\mathcal{T}(G_t)$ is at least $C(\log t)^9t^{3\alpha}$, which by Proposition~\ref{prop:exptriangle} and Markov's inequality occurs with probability at most $1/\log t$. This proves the theorem.
\end{proof}
\bibliography{ref}
\bibliographystyle{plain}
\end{document}